\documentclass[preprint,authoryear,12pt]{elsarticle}

\usepackage{amssymb}

\usepackage{amsthm}
\usepackage[sumlimits]{amsmath}
\usepackage{amscd}
\usepackage{amsfonts}
\usepackage{color,rotating}
\usepackage[margin=2cm]{geometry}
\usepackage{graphicx}
\usepackage{psfrag}
\psfrag{test}{$P1^3$}

\DeclareMathOperator{\diff}{d}

\newtheorem{theorem}{Theorem}

\newtheorem{corollary}[theorem]{Corollary}
\newtheorem{definition}[theorem]{Definition}

\newtheorem{proposition}[theorem]{Proposition}

\def\MM#1{\boldsymbol{#1}}
\newcommand{\pp}[2]{\frac{\partial #1}{\partial #2}}

\usepackage{xspace}

\newcommand{\revised}[1]{{#1}}
\newcommand{\jump}[1]{\left[\!\!\left[ #1 \right]\!\!\right]}
\journal{Journal of Computational Physics}

\begin{document}

\begin{frontmatter}



\title{Embedded discontinuous Galerkin transport schemes with
  localised limiters}


\author[math]{C.~J. Cotter}
\address[math]{Department of Mathematics, Imperial College London,
 South Kensington Campus, London SW7 2AZ}
\author[dmitri]{D. Kuzmin}
\address[dmitri]{Institute of Applied Mathematics, Dortmund University of Technology, Vogelpothsweg 87, D-44227 Dortmund, Germany}

\begin{abstract}
Motivated by finite element spaces used for representation of
temperature in the compatible finite element approach for numerical
weather prediction, we introduce locally bounded transport schemes for
(partially-)continuous finite element spaces. The underlying
high-order transport scheme is constructed by injecting the
partially-continuous field into an embedding discontinuous finite
element space, applying a stable upwind discontinuous Galerkin (DG)
scheme, and projecting back into the partially-continuous space; we
call this an embedded DG transport scheme. We prove that this scheme
is stable in $L^2$ provided that the underlying upwind DG scheme
is. We then provide a framework for applying limiters for embedded DG
transport schemes. Standard DG limiters are applied during the
underlying DG scheme. We introduce a new localised form of
element-based flux-correction which we apply to limiting the
projection back into the partially-continuous space, so that the whole
transport scheme is bounded. We provide details in the specific case
of tensor-product finite element spaces on wedge elements that are
discontinuous P1/Q1 in the horizontal and continuous \revised{P2} in the
\revised{vertical}. The framework is illustrated with numerical tests.
\end{abstract}


\end{frontmatter}

\section{Introduction}

\revised{Recently there has been a lot of activity in the development
  of finite element methods for numerical weather prediction (NWP),
  using continuous (mainly spectral) finite elements as well as
  discontinuous finite elements
  \citep{fournier2004spectral,thomas2005ncar,dennis2011cam,kelly2012continuous,kelly2013implicit,marras2013simulations,brdar2013comparison,bao2015horizontally};
  see \citet{marras2015review} for a comprehensive review. A key
aspect of NWP models is the need for transport schemes that preserve
discrete analogues of properties of the transport equation such as
monotonicity (shape preservation) and positivity; these properties are
particularly important when treating tracers such as
moisture. Discontinuous Galerkin methods can be interpreted as a
generalisation of finite volume methods and hence the roadmap for the
development of shape preserving and positivity preserving methods is
relatively clear (see \citet{cockburn2001runge} for an introduction to
this topic). However, this is not the case for continuous Galerkin
methods, and so different approaches must be used. In the NWP
community, limiters for CG methods have been considered by
\citet{marras2012variational}, who used first-order subcells to reduce
the method to first-order upwind in oscillatory regions, and
\citet{guba2014optimization}, who exploited the monotonicity of the
element-averaged scheme in the spectral element method to build a
quasi-monotone limiter. }

\revised{ In this paper, we address the problem of finding suitable
  limiters for the partially continuous finite element spaces for
  tracers that arise in the framework of compatible finite element
  methods for numerical weather prediction models
  \citep{cotter2012mixed,cotter2014finite,staniforth2013analysis,mcrae2014energy}. Compatible
  finite element methods have been proposed as an evolution of the
  C-grid staggered finite difference methods that are very popular in
  NWP.}  \revised{Within the UK dynamical core ``Gung-Ho'' project},
this \revised{evolution} is being driven by the need to move away from
the latitude-longitude grids which are currently used in NWP models,
since they prohibit parallel scalability
\revised{\citep{staniforth2012horizontal}}. \revised{Compatible finite
  element} methods rely on choosing compatible finite element spaces
for the various prognostic fields (velocity, density, temperature,
etc.), in order to avoid spurious numerical wave propagation that
pollutes the numerical solution on long time scales.  In particular,
in three dimensional models, this calls for the velocity space to be a
div-conforming space such as Raviart-Thomas, and the density space is
the corresponding discontinuous space. Many current operational
forecasting models, such as the Met Office Unified Model
\citep{davies2005new}, use a Charney-Philips grid staggering in the
vertical, to avoid a spurious mode in the vertical. When translated
into the framework of compatible finite element spaces, this requires
the temperature space to be a tensor product of discontinuous
functions in the horizontal and continuous functions in the vertical
(more details are given below). Physics/dynamics coupling then
requires that other tracers (moisture, chemical species \emph{etc.})
also use the same finite element space as temperature.

A critical requirement for numerical weather prediction models is that
the transport schemes for advected tracers do not lead to the creation
of new local maxima and minima, since their coupling back into the
dynamics is very sensitive. In the compatible finite element
framework, this calls for the development of limiters for
partially-continuous finite element spaces. Since there is a
well-developed framework for limiters for discontinuous Galerkin
methods
\revised{\citep{biswas94parallel,burbeau01:problem,cockburn2001runge,hoteit04:newgalerkin,krivodonova04:shockdg,tu05:slope,kuzmin2010vertex,zhang2011maximum}},
in this paper we pursue the three stage approach of (i) injecting the
solution into an embedding discontinuous finite element space at the
beginning of the timestep, then (ii) applying a standard discontinuous
Galerkin timestepping scheme, before finally (iii) projecting the
solution back into the partially continuous space. If the
discontinuous Galerkin scheme is combined with a slope limiter, the
only step where overshoots and undershoots can occur is in the final
projection. In this paper we describe a localised limiter for the
projection stage, which is a modification of element-based limiters
\citep{lohner1987finite,fctools} previously applied to remapping
in \citet{Lohner2008,kuzmin2010failsafe}. This leads to a locally bounded
advection scheme when combined with the other steps.

\revised{
The main results of this paper are:
\begin{enumerate}
\item The introduction of an embedded discontinuous Galerkin scheme which
  is demonstrated to be linearly stable,
\item The introduction of localised element-based limiters to remove
  spurious oscillations when projecting from from discontinuous to
  continuous finite element spaces, which are necessary to make the
  whole transport scheme bounded,
\item When combined with standard limiters for the discontinuous
  Galerkin stage, the overall scheme remains locally bounded,
  addressing the previously unsolved problem of how to limit partially
  continuous finite element spaces that arise in the compatible finite
  element framework.
\end{enumerate}
 Our bounded transport scheme can also be used for continuous finite
 element methods, although other approaches are available that do not
 involve intermediate use of discontinuous Galerkin methods.}

The rest of the paper is structured as follows. The problem is
formulated in Section \ref{sec:formulation}. In particular, more
detail on the finite element spaces is provided in Section
\ref{sec:spaces}. The embedded discontinuous Galerkin schemes are
introduced in Section \ref{sec:embedded}; it is also shown that these
schemes are stable if the underlying discontinuous Galerkin scheme is
stable. The limiters are described in Section \ref{sec:bounded}.
In Section \ref{sec:numerics} we provide some numerical
examples. Finally, in Section \ref{sec:outlook} we provide a summary
and outlook.

\section{Formulation}
\label{sec:formulation}

\subsection{Finite element spaces}
\label{sec:spaces}
We begin by defining the partially continuous finite element spaces
under consideration. In three dimensions, the element domain is
constructed as the tensor product of a two-dimensional horizontal
element domain (a triangle or a quadrilateral) and a one-dimensional
vertical element domain (i.e., an interval); we obtain triangular
prism or hexahedral element domains aligned with the vertical
direction. For a vertical slice geometry in two dimensions (frequently
used in testcases during model development), the horizontal domain is
also an interval, and we obtain quadrilateral elements aligned with
the vertical direction.

To motivate the problem of transport schemes for a partially
continuous finite element space, we first consider a compatible finite
element scheme that uses a discontinuous finite element space for
density.  This is typically formed as the tensor product of the $DG_k$
space in the horizontal (degree $k$ polynomials on triangles or bi-$k$
polynomials on quadrilaterals, allowing discontinuities between
elements) and the $DG_l$ space in the vertical. We consider the case
where the same degree is chosen in horizontal and vertical,
i.e. $k=l$, although there are no restrictions in the framework. We
will denote this space as $DG_k\times DG_k$.

In the compatible finite element framework, the vertical velocity
space is staggered in the vertical from the pressure space; the
staggering is selected by requiring that the divergence (i.e., the
vertical derivative of the vertical velocity) maps from the vertical
velocity space to the pressure space. This means that vertical
velocity is stored as a field in $DG_k\times CG_{k+1}$ (where
$CG_{k+1}$ denotes degree $k+1$ polynomials in each interval element,
with $C^0$ continuity between elements). To avoid spurious hydrostatic
pressure modes, one may then choose to store (potential) temperature
in the same space as vertical velocity (this is the finite element
version of the Charney-Phillips staggering).  \revised{Figure
  \ref{fig:com-fem} provides diagrams showing the nodes for these
  spaces in the case $k=1$.}  Details of how to automate the
construction of these finite element spaces within a code generation
framework are provided in \citet{mcrae2014automated}.

\begin{figure}
  \centerline{\includegraphics[width=15cm]{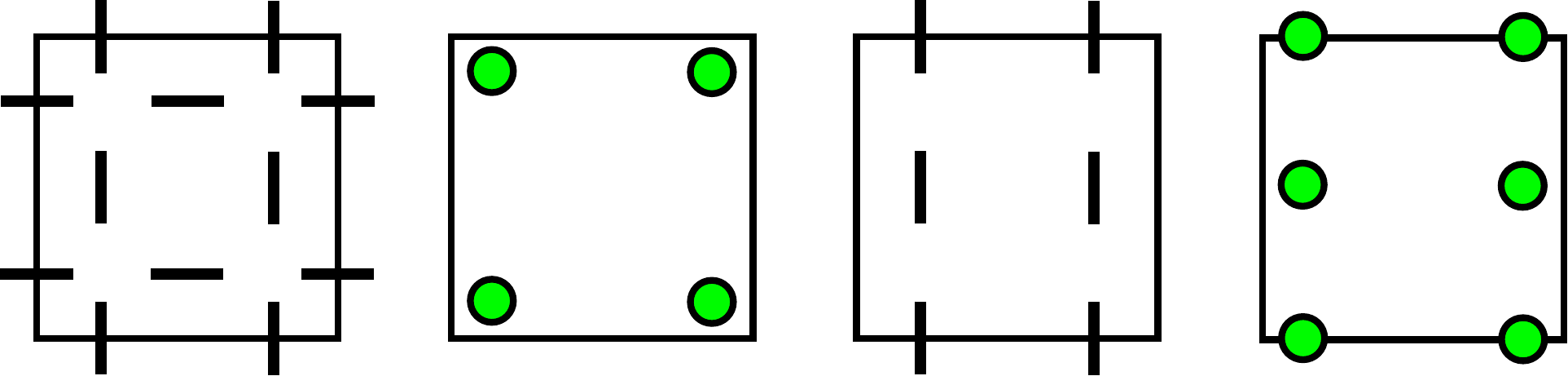}}
  \revised{\caption{\label{fig:com-fem}Diagrams showing the nodes for the
    various spaces in the compatible finite element framework in the
    2D vertical slice case and degree $k=1$. From left to right: the
    velocity space $RT_1$, the pressure space $DG_1\times DG_1$, the
    vertical part of the velocity space $RT_1$, and the temperature
    space $DG_1\times CG_2$. Circles denote scalar nodes, whilst lines
    denote normal components of a vector.}}
\end{figure}

Monotonic transport schemes for temperature are often required,
particularly in challenging testcases such as baroclinic front
generation. Further, dynamics-physics coupling requires that other
tracers such as moisture must be stored at the same points as
temperature; many of these tracers are involved in parameterisation
calculations that involve switches and monotonic advection is required
to avoid spurious formation of rain patterns at the gridscale, for
example. Hence, we must address the challenge of monotonic advection
in the partially continuous $DG_k\times CG_{k+1}$ space.

In this paper, we shall concentrate on the case of $DG_1\times CG_2$.
This is motivated by the fact that we wish to use standard DG upwind
schemes where the advected tracer is simply evaluated on the upwind
side; the lowest order space $DG_0\times CG_1$ leads to a first order
scheme in this case. We may return to higher order spaces in future
work.

\subsection{Embedded Discontinuous Galerkin schemes}
\label{sec:embedded}
\begin{figure}
  \centerline{\includegraphics[width=15cm]{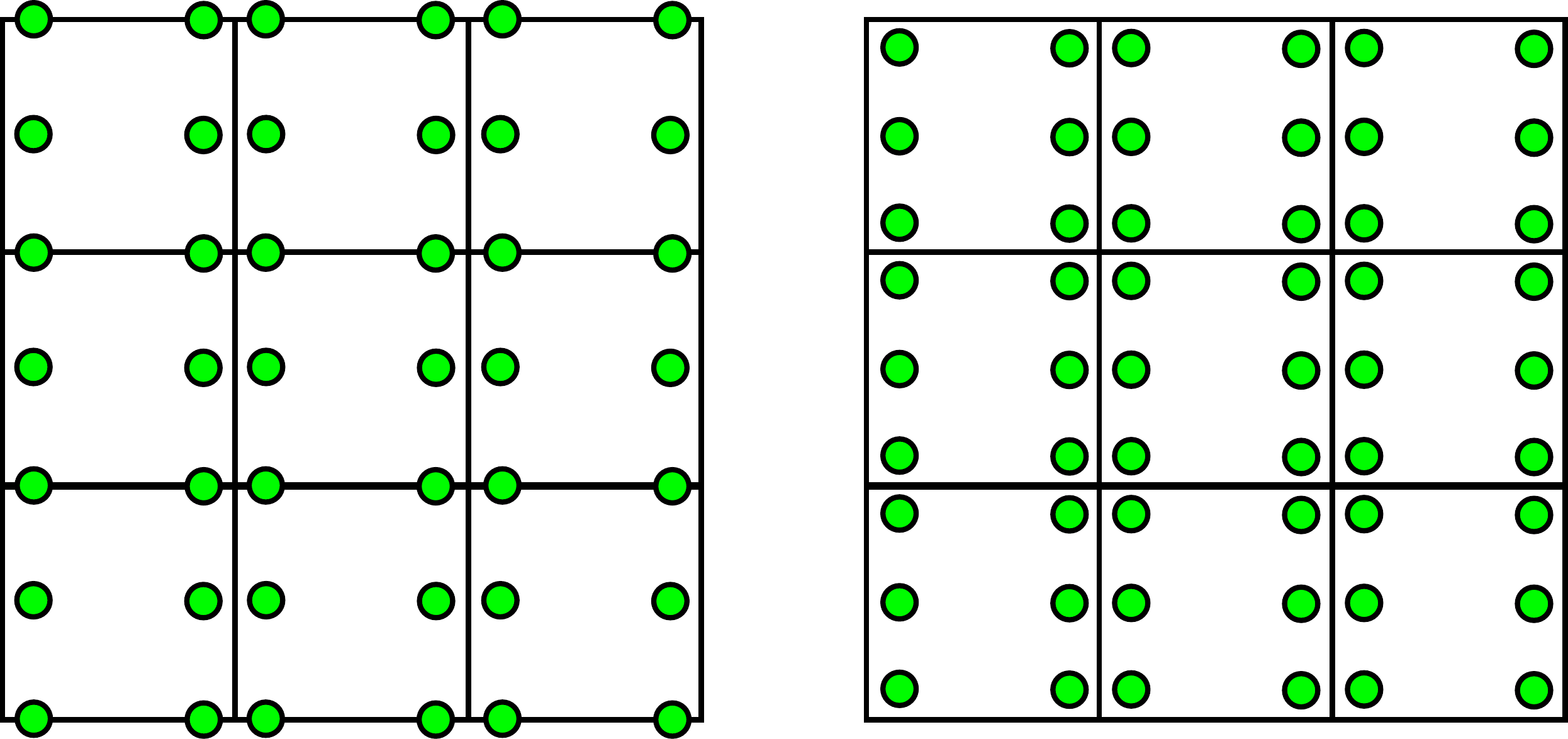}}
\revised{  \caption{\label{fig:vtheta-vthetadg}Diagrams showing the nodes for
    the partially continuous space $V$ and the discontinuous space $\hat{V}$,
    in the case $V=DG_1\times CG_2$ and $\hat{V}=DG_1 \times DG_2$.}}
\end{figure}

\revised{In this section we describe the basic embedded transport
  scheme as a linear transport scheme without limiters.  The scheme,
  which can be applied to continuous or partially-continuous finite
  element spaces, is motivated by the fact that limiters are most
  easily applied to fully discontinuous finite element spaces. We call
  the continuous or partially-continuous finite element space $V$, and
  let $\hat{V}$ be the smallest fully discontinuous finite element
  space containing $V$. A diagram illustrating $V$ and $\hat{V}$ in
  our case of interest, namely the finite element space for
  temperature described in the previous section, is shown in Figure
  \ref{fig:vtheta-vthetadg}.

Before describing the transport scheme, we make a few definitions.
\begin{definition}[Injection operator]
  For $u\in V\subset \hat{V}$, we denote $I:V\to \hat{V}$
  the natural injection operator.
\end{definition}
The injection operator does nothing mathematically except to identify
$Iu$ as a member of $\hat{V}$ instead of just $V$. However, in a
computer implementation, it requires us to expand $u$ in a new basis.
This can be cheaply evaluated element-by-element.
\begin{definition}[Propagation operator]
Let $A:\hat{V}\to\hat{V}$ denote the operator representing the
application of one timestep of an $L^2$-stable discontinuous Galerkin
discretisation of the transport equation.
\end{definition}
For example, $A$ could be the combination of an upwind discontinuous
Galerkin method with a suitable Runge-Kutta scheme.
\begin{definition}[Projection operator]
  For $\hat{u}\in \hat{V}$ we define the projection $P:\hat{V}\to V$ by
  \[
  \langle v,P\hat{u}\rangle = \langle v,\hat{u}\rangle, \quad \forall
  v\in V.
  \]
\end{definition}
In a computer implementation, this requires the inversion of the mass matrix
associated with $V$.

We now combine these operators to construct our embedded discontinuous Galerkin
scheme.
\begin{definition}[Embedded discontinuous Galerkin scheme]
  Let $V\subset\hat{V}$, with injection operator $I$, projection operator $P$
  and propagation operator $A$. Then one step of the embedded discontinuous
  Galerkin scheme is defined as
  \[
  \theta^{n+1} = PAI\theta^n, \quad \theta^n,\theta^{n+1}\in V.
  \]
\end{definition}
The $L^2$ stability of this scheme is ensured by the following 
result.
\begin{proposition}
Let $\alpha>0$ be the stability constant of the 
the propagation operator $A$, such that
\begin{equation}
\|A\| = \sup_{\hat{z}\in \hat{V},\|\hat{z}\|>0}\frac{\|A\hat{z}\|}{\|\hat{z}\|} \leq \alpha,
\end{equation}
where $\|\cdot\|$ denotes
the $L^2$ norm. Then, the stability constant $\alpha^*$ of the
embedded discontinuous Galerkin scheme on $V$ satisfies
$\alpha^*<\alpha$.
\end{proposition}
\begin{proof}
\begin{equation}
\sup_{z\in V,\|z\|>0}\frac{\|PAIz\|}{\|z\|} =
\sup_{z\in V,\|z\|>0}
\frac{\|PAIz\|}{\|Iz\|}
\leq \sup_{\hat{z}\in \hat{V},\|\hat{z}\|>0}
\frac{\|PA\hat{z}\|}{\|\hat{z}\|}
\leq \sup_{\hat{z}\in \hat{V},\|\hat{z}\|>0}\frac{\|A \hat{z}\|}{\|\hat{z}\|} \leq \alpha,
\end{equation}
as required. In the last inequality we used the fact that $\|P\hat{z}\|
\leq \|\hat{z}\|$, which is a consequence of the Riesz representation
theorem.
\end{proof}
\begin{corollary}
  For a given velocity field $\MM{u}$, let $A_{\Delta t}$ denote the
  propagation operator for timestep size $\Delta t$. Let $\Delta t^*$ denote
  the critical timestep for $A_{\Delta t}$, \emph{i.e.},
  \[
  \|A_{\Delta t}\|\leq1, \quad \mbox{ for } \Delta t \leq \Delta t^*.
  \]
  Then, the critical timestep size $\Delta t^\dagger$ for the embedded
  discontinuous Galerkin scheme $PA_{\Delta t}I$ is at least as large
  as $\Delta t^*$.
\end{corollary}
\begin{proof}
  If $\Delta t\leq\Delta t^*$, then
  \[
  \|PA_{\Delta t}I\|\leq \|A_{\Delta t}\| \leq 1,
  \]
  as required.
\end{proof}
Hence, the embedded DG scheme is $L^2$ stable whenever the propagation
operator $A$ is. 

For the numerical examples in this paper, we consider the case
$V=DG_1\times CG_2$ (our temperature space) and $\hat{V}=DG_1\times
DG_2$.  For a given divergence-free velocity field $\MM{u}$, defined
on the domain $\Omega$ and satisfying $\MM{u}\cdot\MM{n}=0$ on the
domain boundary $\partial \Omega$, $A$ represents the application of
one timestep applied to the transport equation
\begin{equation}
  \theta_t = -\MM{u}\cdot\nabla\theta = -\nabla\cdot(\MM{u}\theta),
\end{equation}
discretised using the usual Runge-Kutta discontinuous Galerkin
discretisation (see \citet{cockburn2001runge} for a review). To do
this, first we define $L:\hat{V}\to\hat{V}$ by
\begin{equation}
  \int_\Omega \gamma L\theta \diff x =
  -\Delta t  \int_{\Omega} \nabla\gamma\cdot\MM{u} \theta \diff x
  +\Delta t \int_{\Gamma} \jump{\MM{u}\gamma}\tilde{\theta}\diff S,
\end{equation}
where $\Gamma$ is the set of interior facets in the finite element
mesh, with the two sides of each facet arbitrarily labelled by $+$ and
$-$, the jump operator denotes $\jump{\MM{v}}=\MM{v}^+\cdot\MM{n}^++\MM{v}^-\cdot\MM{n}^-$, and where $\tilde{\theta}$ is the upwind value of $\theta$
defined by 
\[
\tilde{\theta} = \left\{
\begin{array}{cc}
\theta^+ & \mbox{ if } \MM{u}\cdot\MM{n}^+<0, \\
\theta^- & \mbox{ otherwise.}
\end{array}
\right.
\]
Then, the timestepping method is defined by the usual 
3rd order 3 step SSPRK timestepping method
\citep{shu1988efficient},
\begin{align}
  \phi^1 &= \theta^n + L\theta^n, \\
  \phi^2 &= \frac{3}{4}\theta^n + \frac{1}{4}(\phi^1 + L\phi^1), \\
  A\theta^n = \theta^{n+1} &= \frac{1}{3}\theta^n + \frac{2}{3}
  (\phi^2 + L\phi^2).
\end{align}
Since the finite element space $V$
is discontinuous in the horizontal, the projection $P:\hat{V}\to V$
decouples into independent problems to solve in each column
(\emph{i.e.}, the mass matrix for $DG_1\times CG_2$ is column-block
diagonal).
}

\subsection{Bounded transport}
\label{sec:bounded}
Next we wish to add limiters to the scheme. This is done in two
stages. \revised{ First, a slope limiter should be incorporated into the
$\hat{V}$ propagator, $A$; we call the resulting scheme $\tilde{A}$.
A suitable limiter is defined in Section \ref{sec:dg limiter}
 After replacing $A$ with $\tilde{A}$, the only way that the solution
 can generate overshoots and undershoots is after the application of the
 projection $P$.} To control these unwanted
oscillations, we apply a (conservative) flux correction to the 
projection, referred to as flux corrected remapping
\citep{kuzmin2010failsafe}\revised{; this is described in Section \ref{sec:fcr}.
  We denote the flux corrected remapping $\tilde{P}$, and the resulting
  bounded transport scheme may be written as $\theta^{n+1}=\tilde{P}\tilde{A}I$.
}

\subsubsection{Slope limiter \revised{for the propagator $A$}}
\label{sec:dg limiter}
\revised{In principle, any suitable discontinuous Galerkin slope
  limiter can be used in the propagator $A$.}  In this paper we used
the vertex-based slope limiter of \citet{kuzmin2010vertex}. This
limiter is both very easy to implement, and supports a treatment of
the quadratic structure in the vertical.  \revised{Before presenting
  the limiter for $\hat{V}=DG_1 \times DG_2$ (recall that this is the
  space we must use to obtain a transport scheme for our $DG_1\times
  CG_2$ space used for temperature), we first review the concepts in
  the simpler case of $\hat{V}=DG_1\times DG_1$.}  The basic idea for
\revised{$\theta\in\hat{V}=DG_1\times DG_1$} is to write
\begin{equation}
\theta = \bar{\theta} + \Delta \theta,
\end{equation}
where $\bar{\theta}$ is the projection of $\theta$ into $DG_0$,
\emph{i.e.}  in each element $\bar{\theta}$ is the element-averaged
value of $\theta$. Then, for each vertex $i$ in the mesh, we compute
maximum and minimum bounds $\theta_{\max,i}$ and $\theta_{\min,i}$ by
computing the maximum and minimum values of $\bar{\theta}$ over all
the elements that contain that vertex, respectively.  In each element
$e$ we then compute a constant $0\leq \alpha_e\leq 1$ such that the
value of
\begin{equation}
\theta_{\min,i} \leq \theta_e(\MM{x}_i) = \bar{\theta}_e + \alpha_e(\Delta \theta)_e(\MM{x}_i) \leq \theta_{\max,i},
\end{equation}
at each vertex $i$ contained by element $e$. {The optimal value of the correction
factor $\alpha_e$ can be determined using the formula of \citet{barthjesp1989}
\begin{equation}
\alpha_e=\min\limits_{i\in {\cal N}_e}\left\{\begin{array}{ll}
\min\left\{1,\frac{\theta_{\max,i}-\bar \theta_e}{\theta_{e,i}-\bar \theta_e}\right\} &
\mbox{if}\quad \theta_{e,i}-\bar \theta_e>0,\\[0.1cm]
1 &\mbox{if} \quad \theta_{e,i}-\bar \theta_e=0,\\[0.0cm]
\min\left\{1,\frac{\theta_{\min,i}-\bar \theta_e}{\theta_{e,i}-\bar\theta_e}\right\} &
\mbox{if}\quad \theta_{e,i}-\bar \theta_e<0,
\end{array}\right.\label{dglim}
\end{equation}
where ${\cal N}_e$ is the set of vertices of element $e$ and $\theta_{e,i}=
\bar\theta_e+(\Delta\theta)_e(\MM{x}_i)$ is
the unconstrained value of the $DG_1$ shape function at the $i$-th vertex.}

\revised{For our temperature space $DG_1 \times DG_2$ applied to
  numerical weather prediction applications, we assume that we have a
  columnar mesh. This means that the prismatic elements are stacked
  vertically in layers, with vertical sidewalls (but possibly with
  tilted top and bottom faces to facilitate terrain-following meshes,
  so that the elements are trapezia). This allows us to adopt a Taylor
  basis in the vertical, \emph{i.e.} the basis in local coordinates is
  the tensor product of a Taylor basis in the vertical with a Lagrange
  basis in the horizontal.  We write}
\begin{equation}
\label{eq:2d-decomp}
\theta = {\bar{\theta}} + 
(\theta_1 - \bar{\theta}) + (\theta-\theta_1),
\end{equation}
\revised{
  where $\theta_1\in DG_1\times DG_1$, and satisfies the following conditions:
  \begin{enumerate}
  \item $\bar{\theta}_1=\bar{\theta}$,
  \item $\pp{\theta_1}{z}$ and $\pp{\theta}{z}$ take the same values
    along the horizontal element midline in local coordinates.
  \end{enumerate}
}\noindent Then, $\pp{\theta}{z} \in DG_1\times
DG_1$ whilst $\pp{\theta_1}{z} \in DG_1 \times DG_0$. 

First, we limit the quadratic component in the vertical (the third term in
Equation \eqref{eq:2d-decomp}), performing the following steps.
\begin{enumerate}
\item In each element, compute $\pp{\theta_1}{z}$, and
  \revised{evaluate the derivative at the horizontal cell midline to
    obtain $\overline{\pp{\theta_1}{z}} \in DG_1 \times DG_0$.}  If
  the quadratic component $\theta-\theta_1$ is limited to zero then
  $\pp{\theta_1}{z}$ will become equal to
  $\overline{\pp{\theta_1}{z}}$.
\item In each column, at each vertex $i$, compute bounds
  $\pp{\theta}{z}|_{\min,i}$ and $\pp{\theta}{z}|_{\max,i}$ by taking
  the maximum value of $\overline{\pp{\theta_1}{z}}$ at that vertex in the
  elements \revised{sharing that vertex} in the column.
\item In each element, compute element correction factors $\alpha_{1,e}$ 
  according to
\begin{equation}
\alpha_{1,e}=\min\limits_{i\in {\cal N}_e}\left\{\begin{array}{ll}
\min\left\{1,\frac{\pp{\theta}{z}|_{\max,i}-\overline{\pp{\theta}{z}}|_{e,i}}
{\pp{\theta}{z}_{e,i}-\overline{\pp{\theta}{z}}|_{e,i}}\right\} &
\mbox{if}\quad \pp{\theta}{z}|_{e,i}-\overline{\pp{\theta}{z}}|_{e,i}>0,\\[0.1cm]
1 &\mbox{if} \quad \revised{\pp{\theta}{z}|_{e,i}}-\overline{\pp{\theta}{z}}|_{e,i}=0,\\[0.0cm]
\min\left\{1,\frac{\pp{\theta}{z}|_{\min,i}-\overline{\pp{\theta}{z}}|_{e,i}}
{\pp{\theta}{z}|_{e,i}-\overline{\pp{\theta}{z}}|_{e,i}}\right\} &
\mbox{if}\quad \pp{\theta}{z}|_{e,i}-\overline{\pp{\theta}{z}}|_e<0.
\end{array}\right.
\end{equation}
\end{enumerate}
This approach can also be extended to meshes in spherical geometry for
which all side walls are parallel to the radial
direction\footnote{Such meshes arise when terrain following grids are
  used in spherical geometry.}, having replaced $\pp{}{z}$ by the
radial derivative.

Second, we apply the vertex-based limiter to the {$DG_1\times DG_1$}
component $\theta_1$, obtaining limiting constants $\alpha_0$. We then
finally evaluate
\begin{equation}
\theta \mapsto \theta = \bar{\theta} + \alpha_0(\theta_1 - \bar{\theta}) + \alpha_1(\theta-\theta_1).
\end{equation}
To reduce diffusion of smooth extrema, it was recommended in
\citet{kuzmin2010vertex} to recompute the $\alpha_0$ coefficients
according to
\begin{equation}
\alpha_{0,e} \mapsto \max(\alpha_{0,e},\alpha_{1,e}).
\end{equation}
However, this does not work in the case of $DG_1 \times DG_2$ since
there is no quadratic component in the horizontal direction, and hence
nonsmooth extrema in the horizontal direction will not be detected.
{A possible remedy is to use $\alpha_{0,e}$ for the horizontal
  gradient and $\max(\alpha_{0,e},\alpha_{1,e})$ for the vertical
  gradient or to limit the directional derivatives separately using an
  anisotropic version of the vertex-based slope limiter
  (\citet{dg_anisotropic}).}

\revised{This limiter is applied to the input to $\tilde{A}$ and after
  each SSPRK stage, to ensure that no new maxima or minima appear in
  the solution over the timestep.}

\subsubsection{Flux corrected remapping}
\label{sec:fcr}
\revised{The final step of the embedded DG scheme is the projection
  $P$ of the DG solution (which we denote here as $\hat{\theta}$) back
  into $V$. We obtain a high-order, but oscillatory solution, which we
  denote $\theta^H$.  To obtain a bounded solution, we introduce a
localised element-based limiter that blends $\theta^H$ with a
low-order bounded solution ${\theta}^L$, such that high-order
approximation is preserved wherever overshoots and undershoots are not
present.}

\revised{First, we must obtain the low-order bounded solution. Using
  the Taylor basis, we remove the quadratic part of $\hat{\theta}$,
  to obtain $\tilde{\theta}\in DG_1\times DG_1$. A low-order bounded solution can then be obtained 
by applying a lumped mass projection,}
\begin{equation}
\revised{M_i \theta_i^L = \int_\Omega \phi_i\tilde{\theta}\diff x
= \sum_{k=1}^m Q_{ik}\tilde{\theta}_k, \quad i=1,\ldots,m, }
\label{l2lumped}
\end{equation}
where the lumped mass $M$ is defined by 
\begin{equation}
M_i = \int_\Omega \phi_i \diff x,
\end{equation}
the projection matrix $Q$ is defined by
\revised{
\begin{equation}
Q_{ik} = \int_\Omega \phi_i\psi_k\diff x, \quad
i=1,\ldots,n,\, k=1,\ldots,\revised{m,}
\end{equation}}
$\{\psi_i\}_{i=1}^\revised{m}$ is a \revised{Lagrange} basis for
$DG_1\times DG_1$ \revised{and $\{\phi_i\}_{i=1}^n$ is a
  \revised{Lagrange} basis for $DG_1\times CG_2$. }

The lumped mass $M$ and projection matrix $Q$ both have strictly
positive entries. This means that for each $1\leq i \leq n$, the basis
coefficient $\theta^L_i$ is a weighted average of values of
\revised{$\tilde{\theta}$} coming from elements that lie in $S(i)$, the support of
$\phi_i$. The weights are all positive, and hence the value of
$\theta^L_i$ is bounded by the maximum and minimum values of
\revised{$\tilde{\theta}$} in $S(i)$. Hence, no new maxima or minima appear in the
low order solution.

\revised{Next, we combine the low order and high order solutions
  element-by-element, in a process called element-based flux
  correction. Element based flux correction was introduced in \citep{lohner1987finite} and formulated for conservative remapping in
  {\citep{Lohner2008,kuzmin2010failsafe}}. Here, we use a new localised
  element-based formulation, where element contributions to the low and
  high order solutions are blended locally and then assembled.}

\revised{To formulate the element-based limiter, we note that the} {consistent mass counterpart of (\ref{l2lumped}) is given by
\begin{equation}
\sum_{j=1}^nM_{ij}\theta_j^H = \int_\Omega \phi_i\hat{\theta}\diff x, \quad i=1,\ldots,n, 
\label{l2consistent}
\end{equation}
where 
\begin{equation}
M_{ij} = \int_\Omega \phi_i\phi_j \diff x.
\end{equation}
}

First, by repeated addition and subtraction of terms, we write (with no implied sum over the
index $i$)
\revised{\begin{equation}
    M_i\theta^H_i = M_i\theta^L_i + f_i
  \end{equation}
  where
  \begin{align}
    f_i&=M_i\theta^H_i-\sum_jM_{ij}\theta^H_j 
    + M_i\theta^L_i + \sum_jM_{ij}\theta^H_j,\\
     &= M_i\theta^H_i-\sum_jM_{ij}\theta^H_j 
    + \int_{\Omega} \phi_i(\hat{\theta}-\tilde{\theta})\diff x.
\end{align}}
This can be decomposed into elements to obtain
\revised{
  \begin{equation}
    M_i\theta^H_i = \sum_e\left(M_i^e\theta^L_i + f_i^e\right), \quad
    f_i^e = M_i^e\theta^H_i-\sum_jM_{ij}^e\theta^H_j 
    + \int_e \phi_i(\hat{\theta}-\tilde{\theta})\diff x,
  \end{equation}
  where
\begin{equation}
  M_i^e = \int_e \phi_i \diff x,  \quad \mbox{ and }
    M_{ij}^e = \int_e \phi_i\phi_j \diff x.
\end{equation}}
\revised{Importantly, the contributions $f^e_i$ of element $e$
  to its vertices sum to zero, since
\begin{align}
  \nonumber
  \sum_{i=1}^nf_i^e &=\sum_{i=1}^nM_i^e\theta^H_i-\sum_{i=1}^n\sum_{j=1}^nM_{ij}^e\theta_j^H + \int_e\underbrace{\sum_{i=1}^n\phi_i}_{=1}(\hat{\theta}-\tilde{\theta})\diff x, \\
  &= 
  \underbrace{\sum_{i=1}^nM_i^e\theta^H_i-\sum_{j=1}^nM_{j}^e\theta_j^H}_{=0}
   + \underbrace{\int_e(\hat{\theta}-\tilde{\theta})\diff x}_{=0}=0. 
\end{align}
It follows that the total mass of the solution remains unchanged
(i.e., $\sum_{i=1}^nM_i\theta_i^H=\sum_{i=1}^nM_i\theta_i^L$) if all
contributions of the same element are reduced by the same amount.  }

We can then choose element limiting constants $\alpha^e$
to get
\begin{equation}
M_i\theta^H_i = \sum_e\left(M_i^e\theta^L_i + \alpha_e f_i^e\right),
\end{equation}
where $0\leq\alpha_e\leq 1$ is a limiting constant for each element
which is chosen to satisfy {vertex bounds obtained from the nodal
values of $\hat{\theta}$.}

{
The bounds in each vertex are obtained as follows. First element
bounds $\theta_{\max}^e$ and $\theta_{\min}^e$ are obtained from 
$\hat{\theta}$ by maximising/minimising over the vertices 
of element $e$. Then for each vertex $i$, maxima/minima are
obtained by maximising/minimising over the elements
containing the vertex:
\begin{equation}
\theta_{\max,i}=\max_e\theta_{\max}^e,\quad
\theta_{\min,i}=\min_e\theta_{\min}^e.
\end{equation}
The correction factor $\alpha_e$ is chosen so as to enforce the
local inequality constraints
\begin{equation}
M_i^e\theta_{\min,i}\le M_i^e\theta_i^L+\alpha_ef_i^e\le
M_i^e\theta_{\max,i}
\end{equation}
Summing over all elements, one obtains the corresponding global estimate
\begin{equation}
M_i\theta_{\min,i}\le M_i\theta_i^L+\sum_e\alpha_ef_i^e\le
M_i\theta_{\max,i},
\end{equation}
which proves that the corrected value $\theta_i^C:=
\theta_i^L+\frac{1}{M_i}\sum_e\alpha_ef_i^e$ is bounded by
$\theta_{\max,i}$ and $\theta_{\min,i}$.}

To enforce the above maximum principles, we limit
the element contributions $f_i^e$ using 
\begin{equation}
\alpha_e=\min\limits_{i\in {\cal N}_e}
\left\{\begin{array}{ll}
\min\left\{1,\frac{M_i^e( \theta_{\max,i} -\theta^L_i)}{f_{i}^{e}}\right\} &
\mbox{if}\quad f_{i}^{e}>0,\\[0.1cm]
1 &\mbox{if} \quad f_{i}^{e}=0,\\[0.0cm]
\min\left\{1,\frac{M_i^e (\theta_{\min,i}- \theta^L_i)}{f_{i}^{e}}\right\} &
\mbox{if}\quad f_{i}^{e}<0.
\end{array}\right.
\end{equation}
This definition of $\alpha_e$ corresponds to a localised version of
the element-based multidimensional FCT limiter
(\citep{lohner1987finite,fctools}) and has the same structure as
formula (\ref{dglim}) for the correction factors that we used to
constrain the $DG_1$ approximation.  \revised{A further advantage of
  the localised formulation is that the limited fluxes can be built
  independently in each element, before assembling globally and
  dividing by the global lumped mass by iterating over nodes.}

\section{Numerical Experiments}
\label{sec:numerics}

\begin{figure}
  \centerline{\includegraphics[width=14cm]{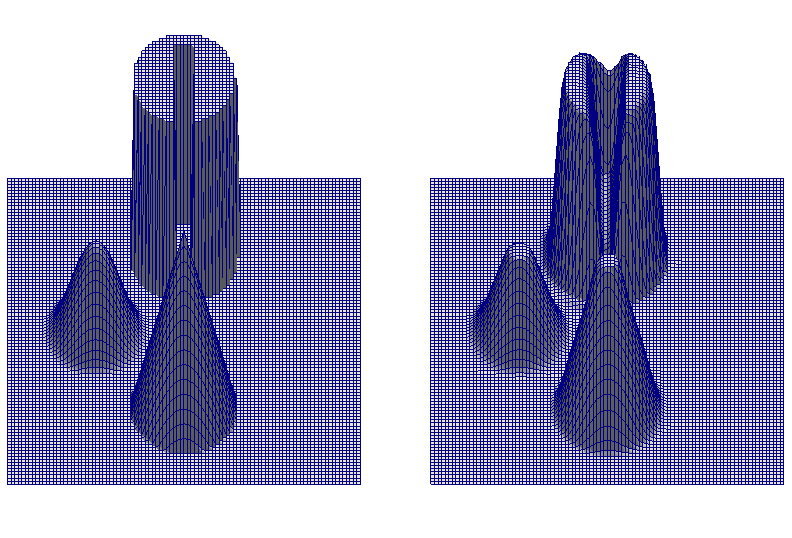}}
  \caption{\label{fig:solid} Solid body rotation test (Section
  \ref{sec:solid}) on a 100$\times$100 grid. Solution is interpolated
  to a $DG_1$ discontinuous field before plotting. Left: initial
  condition. Right: Solution after one rotation. The solution is free
  from over- and undershoots and exhibits comparable numerical
  dissipation to discontinuous Galerkin methods combined with
  limiters.}
\end{figure}

In this section, we provide some numerical experiments demonstrating the
localised limiter for embedded Discontinuous Galerkin schemes.

\subsection{Solid body rotation}
\label{sec:solid}
In this standard test case, the transport equations are solved in the
unit square $\Omega=(0,1)^2$ with velocity field
$\MM{u}(x,y)=(0.5-y,x-0.5)$, \emph{i.e.} a solid body rotation in
anticlockwise direction about the centre of the domain, so that the
exact solution at time $t=2\pi$ is equal to the initial condition. The
initial condition is chosen to be the standard hump-cone-slotted
cylinder configuration defined in \citet{leveque96:highres}, and
solved on a regular mesh with element width $h=1/100$ \revised{and Courant number 0.3}. The result,
shown in Figure \ref{fig:solid}, is comparable with the result for the
\revised{$DG_1$} discontinuous Galerkin vertex-based limiter shown in
Figure 2 of \citet{kuzmin2010vertex}; it is free from over- and
undershoots and exhibits a similar amount of numerical
diffusion. \revised{It is also hard to distinguish between the
  $x$-direction, where the finite element space is discontinuous, and
  the $y$-direction, where the space is continuous.  This suggests
  that we have achieved our goal of constructing a limited transport
  scheme for our partially-continuous finite element space.}

\revised{\subsection{Advection of a discontinous function with curvature}
\label{sec:curvybump}
In this test case, the transport equations are solved in the unit
square $\Omega=(0,1)^2$ with velocity field $\MM{u}=(1,0)$,
\emph{i.e.}  steady translation in the $x$-direction (which is the
direction of discontinuity in the finite element space). The initial
condition is
\begin{equation}
\theta = \left\{
\begin{array}{cc}
4y(1-y) + 1 & \mbox{ if } 0.2<x<0.4, \\
4y(1-y) & \mbox{ otherwise. }
\end{array}\right.
\end{equation}
This test case is challenging because the height of the ``plateau''
next to the continuity varies as a function of $y$ (\emph{i.e.}, in
the direction tangential to the discontinuity); this means that the
behaviour of the limiter is more sensitive to the process of obtaining
local bounds.

The equations are integrated until $t=0.4$ in a $100\times100$ square
grid and Courant number 0.3. The results are showing in Figure
\ref{fig:curvybump}. One can see qualitatively that the degradation in
the solution due to the limiter and numerical errors is not too great.
}
\begin{figure}
\centerline{\includegraphics[width=8cm]{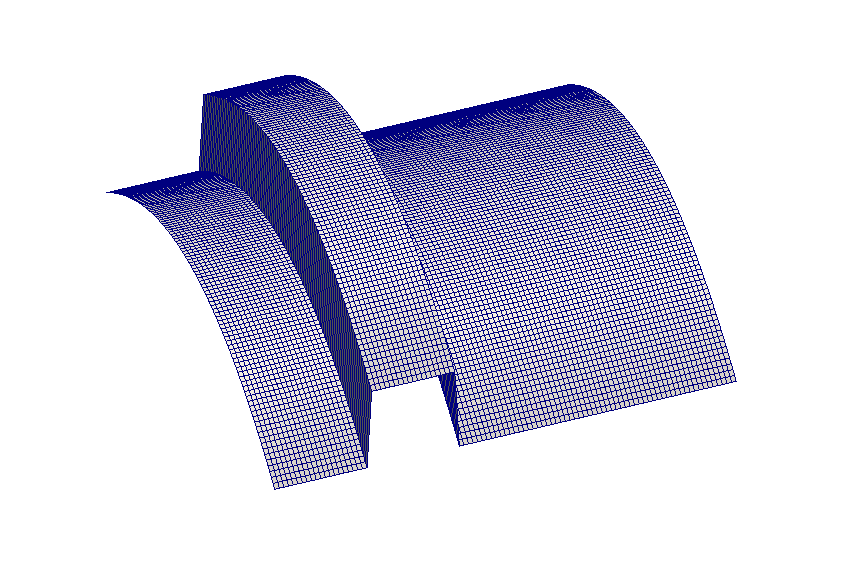}
\includegraphics[width=8cm]{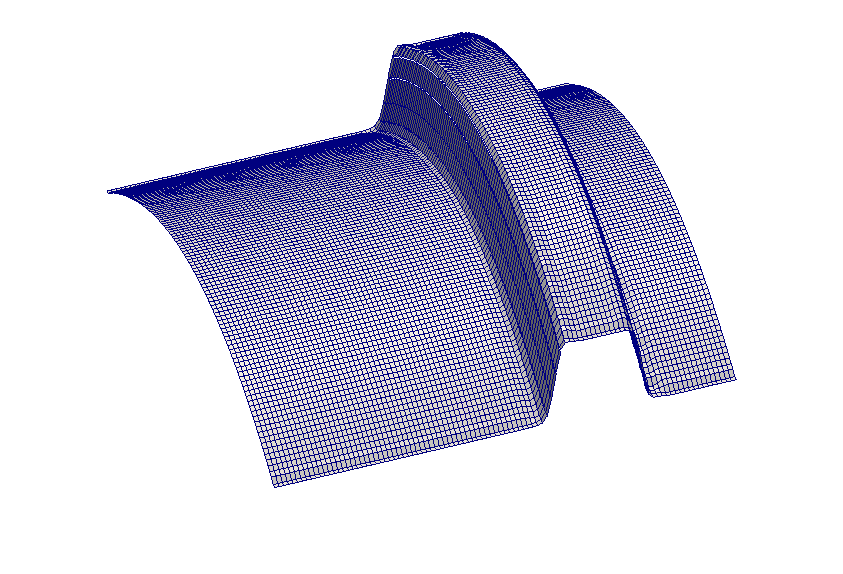}}
\revised{\caption{\label{fig:curvybump}Results from the test case in Section
  \ref{sec:curvybump}. {\bfseries Left:} The initial condition. {\bfseries Right:} The solution at time $t=0.4$.}}
\end{figure}

\subsection{Convergence test: deformational flow}
\label{sec:convergence}
In this test, we consider the advection of a smooth function by a
deformational flow field that is reversed so that the function at time
$t=1$ is equal to the initial condition. As is standard for this type
of test, we add a translational component to the flow and solve the
problem with periodic boundary conditions to eliminate the possibility
of fortuitous error cancellation due to the time reversal.

The transport equations are solved in a unit square, with periodic boundary 
conditions in the $x$-direction. The initial condition is
\[
\theta(\MM{x},0) = 0.25(1+\cos(r)), 
\quad r = \min\left(0.2,\sqrt{(x-0.3)^2 + (y-0.5)^2}/0.2\right),
\]
and the velocity field is
\[
\MM{u}(\MM{x},t) = \left(1-5(0.5-t)\sin(2\pi(x-t))\cos(\pi y),
5(0.5-t)\cos(2\pi(x-t))\sin(\pi y)\right),
\]
\revised{where $\MM{x}=(x,y)$}. The problem was solved on a sequence
of regular meshes with square elements \revised{at fixed timestep $\Delta t
=0.000856898$}, and the $L^2$ error was computed. A plot of the errors
is provided in Figure \ref{fig:convergence}. As expected, we obtain
second-order convergence (the quadratic space in the vertical does not
enhance convergence rate because the full two-dimensional quadratic
space is not spanned).

\begin{figure}
  \centerline{\includegraphics[width=12cm]{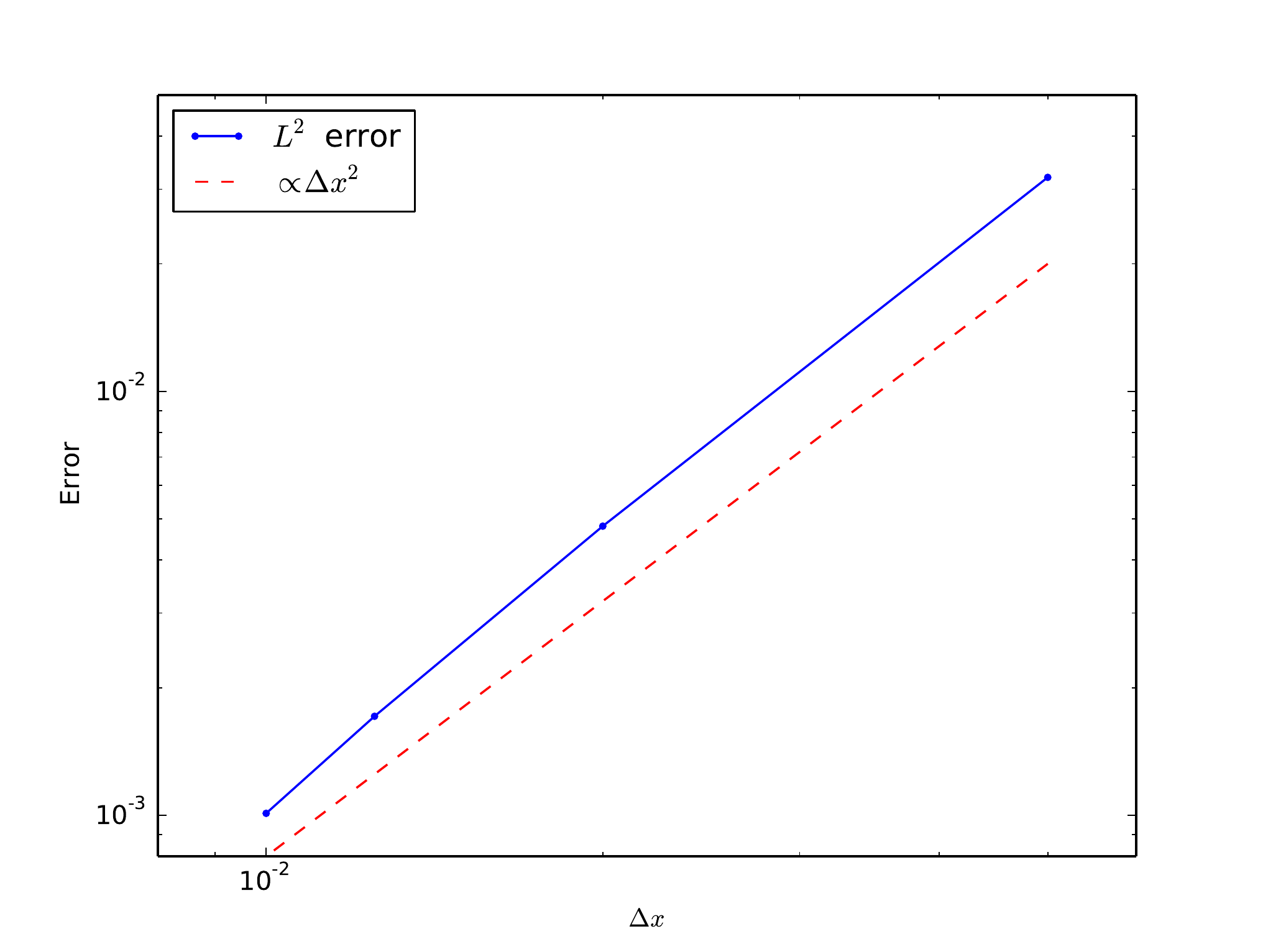}}

  \centerline{\begin{tabular}{|c|c|}
    \hline
    $\Delta x$ & $L^2$ error \\
    \hline
    0.05 & 0.0319911 \\
    0.02 &   0.0048104 \\
    0.0125 & 0.0017125 \\
    0.01 &   0.0010108 \\
    \hline
  \end{tabular}}
\caption{\label{fig:convergence}Convergence plot for deformational
  flow experiment (Section \ref{sec:convergence}) showing second order
  convergence, and a table of error values.}
\end{figure}

\section{Summary and Outlook}
\label{sec:outlook}
In this paper we described a limited transport scheme for
partially-continuous finite element spaces. Motivated by numerical
weather prediction applications, where the finite element space for
temperature and other tracers is imposed by hydrostatic balance and
wave propagation properties, we focussed particularly on the case of
tensor-product elements that are continuous in the vertical direction
but discontinuous in the horizontal. However, the entire methodology
applies to standard \revised{$C_0$} finite element spaces. The
transport scheme was demonstrated in terms of convergence rate on
smooth solutions and dissipative behaviour for non-smooth solutions in
some standard testcases.

Having a bounded transport scheme for tracers is a strong requirement
for numerical weather prediction algorithms; the development of our
scheme advances the practical usage of the compatible finite element
methods described in the introduction. The performance of this
transport scheme applied to temperature in a fully coupled atmosphere
model will be evaluated in 2D and 3D testcases as part of the ``Gung
Ho'' UK Dynamical Core project in collaboration with the Met Office.
In the case of triangular prism elements we anticipate that it may be
necessary to modify the algorithm above to limit the time derivatives
as described in \citet{kuzmin2013slope}.

A key novel aspect of our transport scheme is the localised
element-based FCT limiter. This limiter has much broader potential
for use in FCT schemes for continuous finite element spaces, which
will be explored and developed in future work.

\section{Acknowledgements}
Colin Cotter acknowledges funding from NERC grant NE/K006789/1. Dmitri
Kuzmin acknowledges funding from DFG grant KU 1530/12-1.

\bibliographystyle{elsarticle-harv}
\bibliography{paper}

\end{document}